\documentclass{article}
\usepackage{amssymb,amsmath, amsthm, tikz, braids, mathtools,comment}
\usepackage{graphicx}
\usepackage{lipsum}

\newcounter{strands}
\pgfkeyssetvalue{/tikz/braid height}{3cm}
\pgfkeyssetvalue{/tikz/braid width}{1cm}
\pgfkeyssetvalue{/tikz/braid start}{(0,0)}
\pgfkeyssetvalue{/tikz/braid colour}{black}
\pgfkeys{/tikz/strands/.code={\setcounter{strands}{#1}}}

\usepackage[framemethod=TikZ]{mdframed}

\usepackage{graphicx} 
\usepackage{caption} 
\usepackage{subcaption} 
\usepackage{float} 

\usetikzlibrary{intersections}
\usetikzlibrary{knots}
\usetikzlibrary{positioning}

\usetikzlibrary{decorations.markings}
\usepgfmodule{decorations}
\usetikzlibrary{arrows, arrows.meta}

\usepackage{pgfplots}

\def\tr{\triangleright}
\def\otr{ \,\overline{\triangleright} \, }

\newtheorem{theorem}{Theorem}

\newtheorem{proposition}[theorem]{Proposition}

\theoremstyle{definition}
\newtheorem{example}{Example}
\newtheorem{definition}{Definition}
\newtheorem{remark}{Remark}

\newcommand{\Z}{\mathbb{Z}}

\date{}

\title{\Large \textbf{Legendrian Rack Invariants of Legendrian Knots}}

\author{
Jose Ceniceros \footnote{Email: \texttt{jcenicer@hamilton.edu}.}
\and
Mohamed Elhamdadi\footnote{Email: \texttt{emohamed@math.usf.edu}.}
 \and 
 Sam Nelson\footnote{Email: \texttt{knots@esotericka.org}. Partially supported by Simons Foundation collaboration grant 316709.}
 }

\begin{document}

\maketitle

\begin{abstract}
We define a new algebraic structure called \textit{Legendrian racks} or \textit{racks with Legendrian structure}, motivated by the front-projection Reidemeister moves for Legendrian knots. We provide examples of Legendrian racks and use these algebraic structures to define invariants of Legendrian knots with explicit computational examples. We classify Legendrian structures on racks with 3 and 4 elements. We use Legendrian racks to distinguish certain Legendrian knots which are equivalent as smooth knots.

\end{abstract}



\section{\large\textbf{Introduction}}

Racks and quandles are algebraic structures whose axioms were motivated by the Reidemeister moves in knot theory. Quandles were introduced independently by Joyce and Matveev in the 1980s \cite{Joyce, Matveev}; their generalizations known as
racks were introduced in the early 1990s by Fenn and Rourke \cite{FR}.  For oriented non-split links in $S^3$, the fundamental quandle of a link forms a complete invariant up to mirror image.  Quandles have been used to construct invariants of oriented knots and links in many papers over the last few decades.  Quandles have been studied in various contexts:  they have been studied, for example, as  algebraic systems for symmetries in \cite{Taka}, in relation to modules~\cite{Sam}, in relation to the Yang-Baxter equation \cite{ CES1, CN1},  ring theory \cite{EFT} and also in connection with topological spaces in \cite{CES, EM, Rubin}.  

Finite racks and quandles, in particular, give rise to powerful invariants of 
knots, links and other knotted objects (surface-links, handlebody-links, 
spatial graphs) through \textit{counting invariants} and their various 
enhancements. Since quandle colorings are preserved by Reidemeister moves,
the number of quandle colorings of a knot or link diagram is an integer-valued 
invariant. More generally, any invariant of quandle-colored knots and links
defines an invariant called an \textit{enhancement} from which the counting 
invariant can be recovered but which is typically a stronger invariant.
For more details on racks and quandles and their variations see~\cite{ENbook}.  

In \cite{KulPra}, the authors introduced rack invariants of oriented Legendrian knots in $\mathbb{R}^3$ endowed with the standard contact structure. These invariants are not complete but they detect some of the geometric properties in some Legendrian knots such as cusps. In this paper, we define a new algebraic structure called a \textit{Legendrian rack}, motivated by the front-projection Reidemeister moves for Legendrian knots.  We show that the resulting counting invariant distinguishes the unknot and its positive stabilization, the trefoil and its positive stabilization, the trefoil and its negative stabilization and more such pairs. The invariants given in \cite{KulPra} form a special case of our structure, but our invariants are able to distinguish Legendrian knots that are not distinguished by the the invariants in  \cite{KulPra}.  

The paper is organized as follows.  In Section~\ref{Review}, we review the basics of racks and quandles and give some examples.  Section~\ref{Contact} deals with an overview of contact geometry in general and relations to knot theory in particular. In Section~\ref{L}, we define \emph{Legendrian racks} motivated by Reidemeister moves in Legendrian knot theory.  A characterization of $(t,s)$-racks with a certain map being Legendrian racks is given.  This section contains a classification of Legendrian structures on racks with 3 and 4 elements in addition to some other explicit examples.  In Section~\ref{Col} colorings of Legendrian knots by Legendrian racks is used to distinguish certain Legendrian knots.  

\section{\large\textbf{Review of Racks and Quandles}}\label{Review}

We begin with a definition from \cite{FR}.

\begin{definition}
A \textit{rack} is a set $X$ with two binary operations $\tr$ and
$\otr$ satisfying for all $x,y,z\in X$
\begin{list}{}{}
\item[\textup{(i)}]{$(x\tr y)\otr y = x = (x\otr y)\tr y$} and
\item[\textup{(ii)}]{$(x\tr y)\tr z=(x\tr z)\tr(y\tr z)$}.
\end{list}
A rack which further satisfies $x\tr x=x$ for all $x\in X$ is a 
\textit{quandle}.
\end{definition}

\begin{example}\label{ex1} Some examples of racks and quandles include:

\begin{itemize}
\item Any group $G$ is a quandle with operation given by conjugation 
\[x \tr y=y^{-1}xy,\]
called the \textit{conjugation quandle} of $G$.
\item Any group $G$ is a quandle with operation \[x \tr y=yx^{-1}y\]
called the \emph{core} quandle of $G$. Core quandles are \textit{involutory}, 
i.e., $(x \tr y)\tr y=x, \forall x,y \in G$.
\item Any $\Z[t^{\pm 1}]$-module $X$ is a quandle with operation
\[x \tr y=tx+(1-t)y\] called an {\it  Alexander  quandle}.
\item Any group $G$ with an automorphism $\sigma \in {\rm Aut}(G)$ 
is a quandle with operation 
\[x \tr y=\sigma(xy^{-1}) y\] 
called a {\it generalized Alexander quandle}. When $G$ is abelian this reduces
to the case above.
\item Any $\mathbb{Z}[t^{\pm 1},s]/(s^2-(1-t)s)$-module $V$ is a rack
with rack operations
\[x\tr y=tx+sy\]
known as a \textit{$(t,s)$-rack.} Alexander quandles are $(t,s)$-racks with
$s=1-t$.
\end{itemize}
\end{example}

Quandles and racks are of interest in knot theory because they can be used to 
define an easily computable family of knot and link invariants known as 
\textit{counting invariants} or \textit{coloring invariants}. Given a finite
quandle $X$, an assignment of an element of $X$ to each arc in an oriented link
diagram $D$ is an \textit{$X$-coloring} of $D$ if at every crossing we have
the following picture:
\[\includegraphics{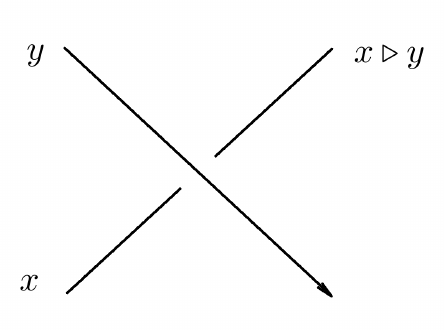}\]
That is, if the overcrossing strand is directed down, then the strand crossing
under from left to right colored by $x$ is acted on by the color on the 
overcrossing strand colored $y$ to become the undercrossing strand colored 
$x\tr y$. If we write $z=x\tr y$, then we can regard the crossing under in 
the opposite direction to be the inverse action by $y$, i.e. we have $z=x\tr y$
crossing under $y$ from right to left to become $x=z\tr^{-1} y$.

It is straightforward to check that Reidemeister moves do not change the number
of $X$-colorings of an oriented link diagram when $X$ is a quandle, and 
blackboard-framed Reidemeister moves do not change the number of $X$-colorings
of a blackboard-framed oriented link diagram. Hence, from any diagram $D$ of
an oriented link, we can compute the \textit{quandle counting invariant} 
$\Phi_X^{\mathbb{Z}}(L)$, i.e. the number of quandle colorings of our diagram 
$D$. This is an integer-valued invariant of oriented knots and links. 

\begin{example}
The trefoil knot below has 9 colorings by the Alexander quandle 
$X=\mathbb{Z}_3[t]/(t-2)$ as one can compute easily from the system of coloring
equations determined by the crossings.
\[\raisebox{-0.9in}{\includegraphics{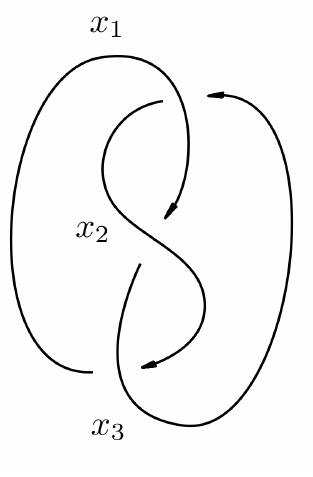}}\quad \quad
\begin{array}{rcl} 
tx_2+(1-t)x_1 & = & x_3 \\
tx_3+(1-t)x_2 & = & x_1 \\
tx_1+(1-t)x_3 & = & x_2 \\
\end{array}
\Rightarrow
\begin{array}{rcl} 
2x_2+2x_1 & = & x_3 \\
2x_3+2x_2 & = & x_1\\
2x_1+2x_3 & = & x_2.
\end{array}
\]
\end{example}
See \cite{ENbook} for more.

\section{Contact Manifolds and Knot Theory}\label{Contact}

\subsection{Standard Contact Structure on $\mathbb{R}^3$}
In this section we will introduce contact structures and related terminology. The goal of this section is to give an overview of contact geometry, for a more complete description of the theory and for important results the reader is referred to \cite{EN, Etnyre, Etnyre1, Geiges}.

\begin{definition}
\textup{An oriented 2-plane field $\xi$ on a 3-manifold $M$ is called a \textit{contact structure} if for any 1-form defined locally or globally with $\xi = \text{ker}(\alpha)$ satisfies $\alpha \wedge d \alpha \neq 0$. The pair $(M, \xi)$ is called a \textit{contact manifold}.} 
\end{definition}

The condition $\alpha \wedge d\alpha \neq 0$ is known as a totally non-integrability condition. This condition ensures that there is no embedded surface in $M$ which is tangent to $\xi$ on any open neighborhood. In this paper we will restrict our attention to the following contact structure on $\mathbb{R}^3$.

\begin{example}\label{stdcon}
Let $\mathbb{R}^3$ with standard Cartesian coordinates $(x,y,z)$ and the 1-form
\[ \alpha = dz - y dx. \]
We can confirm that the non-integrability condition is met by the following computation
\begin{eqnarray*}
\alpha \wedge d\alpha &=&(dz -ydx) \wedge (-dy \wedge dx) \\
&=&(-dz \wedge dy \wedge dx) + y dx \wedge dy \wedge dx\\
&=& dx \wedge dy \wedge dz.
\end{eqnarray*}
Thus, $\alpha$ is a called the contact form and 
\begin{eqnarray*}
\xi_{std} &=& \mathrm{ker}(\alpha)\\ 
&=& \mathrm{ker}(dz -ydx)\\
&=& \mathrm{Span}\left\lbrace  \frac{\partial}{\partial y}, y \frac{\partial}{\partial z} + \frac{\partial}{\partial x} \right\rbrace
\end{eqnarray*}
is a contact structure on $\mathbb{R}^3$.
\end{example}


\begin{remark}
At any point in the $xz$-plane $\xi$ is horizontal and moving along a ray perpendicular to the $xz$-plane the plane field will always be tangent to this ray and rotate by $\pi/2$ in a right handed manner as  move along the ray.
\end{remark}

Example \ref{stdcon} is commonly referred to as the standard contact structure on $\mathbb{R}^3$. As mentioned above we will restrict our attention to the contact manifold $(\mathbb{R}^3, \xi_{std})$. We will be specifically interested in $1$-dimensional submanifolds in $(\mathbb{R}^3, \xi_{std})$. 

\subsection{Legendrian knots}

We will be considering knots in $(\mathbb{R}^3, \xi_{std})$, which are simple closed curves that respect the geometry imposed by the contact structure. There are two natural ways that knots can respect the geometry imposed by contact structures, therefore, there are two classes of knots: the Legendrian class and the transverse class. We will restrict our attention to Legendrian knots.  This section is not meant to be a complete survey on the subject, for a detail description of knot theory supported in a contact 3-manifold, see \cite{Etnyre, Geiges, Sab}.

We have the following definition from \cite{Geiges}
\begin{definition}
A \textit{Legendrian knot} $L$ in $(\mathbb{R}^3, \xi_{std})$ is a 
smooth embedding of $S^1$ that is always tangent to $\xi_{std}$:
\[ T_x L \in \xi_x, \quad x \in L.\]
\end{definition}
Where $T_xL$ is the tangent space of $L$ at the point $x$ and $\xi_x$ is the contact plane from the contact structure $\xi_{std}$ at the point $x$.

Two Legendrian knots in $(\mathbb{R}^3, \xi_{std})$ are \textit{Legendrian isotopic} if there is an isotopy through Legendrian knots between the two knots. A Legendrian knot can be parameterized by an embedding $\phi: S^1 \rightarrow \mathbb{R}^3$ defined by $\phi(\theta) = ( x(\theta), y(\theta), z (\theta))$. A parametrization of $L$ will induce an orientation on $L$, therefore, we can consider \textit{oriented Legendrian knot} by choosing the orientation induced by $\phi$. Studying the Legendrian knot in $\mathbb{R}^3$ is difficult, therefore, it is common to study projections of $L$ in $\mathbb{R}^2$. We will focus on the projection known as the \textit{front projection}. Before we describe the front projection of $L$, we should note that since $\phi$ is a parametrization of $L$ and  $\xi = ker(dz - y dx)$, therefore, in order for $L$ to be tangent to the contact planes $\phi$ must satisfy the following:
\begin{equation}
z'(\theta) - y(\theta) x'(\theta) = 0. 
\end{equation}  
Let $\Pi : \mathbb{R}^3 \rightarrow \mathbb{R}^2$ defined by $(x, y, z) \mapsto (x,z)$. The image of $L$ under $\Pi$ is the \textit{front projection} of $L$. If $\phi$ is a parametrization of $L$, then 
\[  \phi_{\Pi} : S^1 \rightarrow \mathbb{R}^2 \]
defined by 
\[ \theta \mapsto (x(\theta), z(\theta)) \]
is a parametrization of the image of $L$ under $\Pi$. From equation (1) we get $y(\theta) = \frac{z'(\theta)}{x'(\theta)}$ provided that $(x(\theta), z(\theta))$ does not have vertical tangencies. We can summarize the conditions on a front projection for a Legendrian knot by
\begin{enumerate}
\item $K$ has no vertical tangencies,
\item the only non-smooth points are cusps,
\item at each crossing the slope of the over crossing is smaller than the undercrossing.
\end{enumerate}

Two Legendrian knots $L_1$ and $L_2$ are Legendrian isotopic if and only if their front projections are related by a sequence of \textit{Legendrian Reidemeister moves} listed below as well as the rotation of each diagram by 180 degrees about all the coordinate axes.
\vspace{1cm}
\[
\begin{tabular}{lll}
\begin{tikzpicture}[domain=-2:2, scale=1, knot gap=9, transform canvas={scale=0.65}]
\draw [knot=black,line width=0.50mm](1.732,1) to  [out=left, in=45, looseness=0.7] (-1.732,-1);
\draw [knot=black,line width=0.50mm](-1.732,1) to  [out=right, in=135, looseness=0.7] (1.732,-1);
\draw [knot=black, knot gap=0,line width=0.50mm](-1.732,1) to [out=right, in=left, looseness=0.7] (0,1.5) to [out=right, in=left, looseness=0.7] (1.732,1);
\end{tikzpicture}
\hspace{2cm}
&
\begin{tikzpicture}[scale=.9, transform shape]
\draw[>=triangle 45, <->] (0,0) -- (2,0);
\end{tikzpicture}
\hspace{2cm}
&
\begin{tikzpicture}[domain=-2:2, scale=1, knot gap=9,transform canvas={scale=0.65}]
\draw [knot=black,line width=0.50mm](-2,0) to [out=45, in=left, looseness=1] (0,1) to [out=right, in=135, looseness=1] (2,0);
\end{tikzpicture}
\end{tabular}
\]
\vspace{2cm}
\[
\begin{tabular}{lll}
\begin{tikzpicture}[domain=-2:2, scale=1, knot gap=9, transform canvas={scale=0.65}]
\draw [knot=black,line width=0.50mm]
(-1,1.9365) to [out=right, in=left, looseness=1] (2,0) to [out=left, in=right, looseness=1] (-1,-1.9365);
\draw [knot=black,line width=0.50mm]
(0.2,1.9365) to [out=down, in=up, looseness=0] (1,-1.9365) ;
\end{tikzpicture}
\hspace{2cm}
&
\begin{tikzpicture}[scale=.9, transform shape]
\draw[>=triangle 45, <->] (0,0) -- (2,0);
\end{tikzpicture}
\hspace{2cm}
&
\begin{tikzpicture}[domain=-2:2, scale=1, knot gap=9, transform canvas={scale=0.65}]
\draw [knot=black,line width=0.50mm]
(-1.5,1.9365) to [out=right, in=left, looseness=1] (0.5,0) to [out=left, in=right, looseness=1] (-1.5,-1.9365);
 \draw [knot=black,line width=0.50mm]
(0.5,1.9365) to [out=down, in=up, looseness=0] (1.5,-1.9365) ;
\end{tikzpicture}
\end{tabular}
\]
\vspace{2cm}
\[
\begin{tabular}{lll}
\begin{tikzpicture}[domain=-2:2, scale=1, knot gap=9, transform canvas={scale=0.65}]
\draw [knot=black,line width=0.50mm]
(2,2) to  [out=down, in=up, looseness=0] (-2,-2);
\draw [knot=black,line width=0.50mm]
(2,0) to  [out=left, in=right, looseness=1] (0,1) to  [out=left, in=right, looseness=1] (-2,0)  ;
\draw [knot=black,line width=0.50mm]
(-2,2) to  [out=down, in=up, looseness=0] (2,-2);
\end{tikzpicture}
\hspace{2cm}
&
\begin{tikzpicture}[scale=.9, transform shape]
\draw[>=triangle 45, <->] (0,0) -- (2,0);
\end{tikzpicture}
\hspace{2cm}
&
\begin{tikzpicture}[domain=-2:2, scale=1, knot gap=9, transform canvas={scale=0.65}]
\draw [knot=black,line width=0.50mm]
(2,2) to  [out=down, in=up, looseness=0] (-2,-2);
\draw [knot=black,line width=0.50mm]
(2,0) to  [out=left, in=right, looseness=1] (0,-1) to  [out=left, in=right, looseness=1] (-2,0)  ;
\draw [knot=black,line width=0.50mm]
(-2,2) to  [out=down, in=up, looseness=0] (2,-2);
\end{tikzpicture}
\end{tabular}
\]
\vspace{1cm}

An interesting note about Legendrian knots is that you have different Legendrian knot representatives of a topological knot type. The operation that produces different Legendrian knots of the same topological knot type is called \textit{stabilization}. A stabilization of a Legendrian knot $L$ in a front projection of $L$ can be obtained by removing a strand and replacing it with a zig-zag. We denote a positive stabilization by $S_{+}$ and a negative stabilization by $S_{-}$. 

\vspace{1cm}
\[
\begin{tabular}{lll}
\begin{tikzpicture}[transform canvas={scale=0.65}]
\draw[line width=0.50mm, ->](-7,0) --(-4,0);
\draw[>=triangle 45, ->] (-2,1) -- (3,1.5);
\draw[>=triangle 45 , ->] (-2,-1) -- (3,-1.5);
\draw [line width=0.50mm, ->](4,2) to [out=right, in=left, looseness=.7] (5,2.5) to [out=right, in=left, looseness=.7] (6,2) to [out=left, in=right, looseness=.7] (5,1.5) to [out=right, in=left, looseness=.7] (6,1)to [out=right, in=left, looseness=.7] (7,1.5);
\draw [line width=0.50mm, ->](4,-2) to [out=right, in=left, looseness=.7] (5,-2.5) to [out=right, in=left, looseness=.7] (6,-2) to [out=left, in=right, looseness=.7] (5,-1.5) to [out=right, in=left, looseness=.7] (6,-1)to [out=right, in=left, looseness=.7] (7,-1.5);
\node [draw=none, inner sep = 0] at (0,2) {\Large{$S_{+}$}};
\node [draw=none, inner sep = 0] at (0,-2) {\Large{$S_{-}$}};
\end{tikzpicture}
\end{tabular}
\]
\vspace{1cm}

\begin{example}
The following are two Legendrian isotopic representatives of the unknot:
\vspace{.5cm}
\[
\begin{tabular}{ll}
\begin{tikzpicture}[domain=-3:3, scale=1, knot gap=9, transform canvas={scale=0.65}]
\draw [knot=black, knot gap=0,line width=0.50mm](-3,0) to [out=right, in=left, looseness=0.7] (0,1.5) to [out=right, in=left, looseness=0.7] (3,.75);
\draw [knot=black, knot gap=0,line width=0.50mm](3,.75) to [out=left, in=right, looseness=0.7] (1,0);
\draw [knot=black, knot gap=0,line width=0.50mm](1,0) to [out=right, in=left, looseness=0.7] (3,-.75);
\draw [knot=black, knot gap=0,line width=0.50mm](-3,0) to [out=right, in=left, looseness=0.7] (0,-1.5) to [out=right, in=left, looseness=0.7] (3,-.75);
\end{tikzpicture} \hspace{5cm}
&
\begin{tikzpicture}[domain=-2:2, scale=1, knot gap=9, transform canvas={scale=0.65}]
\begin{knot}[consider self intersections, clip width=9]
\strand[line width=0.50mm](-2,0) to [out=right, in=left, looseness=0.7] (-1,1) to [out=right, in=left, looseness=0.7] (1,-1) to [out=right, in=left, looseness=0.7] (2,0) to [out=left, in=right, looseness=0.7] (1, 1) to [out=left, in=right, looseness=0.7] (-1, -1) to [out=left, in=right, looseness=0.7] (-2, 0);
\end{knot}
\end{tikzpicture}
\end{tabular}
\]
\vspace{.5cm}
\end{example}

The problem of classifying Legendrian knots is a difficult problem, \cite{EN1,EN2}. The first 
invariant is the topological knot type of the Legendrian knot. Legendrian 
knots also come equipped with two invariants known as the \emph{classical} invariants.
The first is the \textit{Thurston-Bennequin number} denoted by $tb$. The 
second invariant is the \textit{rotation number} denoted by $rot$. Both of 
these invariants can be computed directly from front projections, but they 
also have deep relationships to the underlying geometric structure.

A topological knot type is \textit{Legendrian simple} if all Legendrian knots 
in its class are determined up to Legendrian isotopy by their classical 
invariants. Some knot types which are known to be Legendrian simple include the 
unknot, torus knots, and the figure eight knot. Note that the classical 
invariants are not sufficient to classify all Legendrian knots. The 
introduction of finer invariants such as contact homology, Chekanov's DGA, and the GRID invariants have been useful tools in addressing the classification problem \cite{CH, EN3, Oz1}.

\section{Legendrian Racks}\label{L}

We introduce the notion of \textit{Legendrian rack} and we give some examples. We will assign labels to the arcs of a Legendrian knots in the following manner: 

\vspace{.5cm}
\[
\begin{tabular}{lll}
\begin{tikzpicture}[domain=-2:2, scale=1, knot gap=9, transform canvas={scale=0.65}]
\draw [knot=black,line width=0.50mm,
decoration={markings, mark=at position 0.25 with {\arrow[black,ultra thick]{>}}},
decoration={markings, mark=at position 0.76 with {\arrow[black,ultra thick]{>}}}, postaction={decorate}]
(1.5,1.3229) to  [out=-135, in=45, looseness=0] (-1.5,-1.3229) ;
\draw [knot=black,line width=0.50mm,
decoration={markings, mark=at position 0.25  with {\arrow[black,ultra thick]{>}}},
decoration={markings, mark=at position 0.76 with {\arrow[black,ultra thick]{>}}}, postaction={decorate}]
(-1.5,1.3229) to [out=-45, in=135, looseness=0] (1.5,-1.3229);
\node [draw=none, inner sep = 0] at (-1.6,1.6) {\Large{${x}$}};
\node [draw=none, inner sep = 0] at (1.6,1.6) {\Large{$y$}};
\node [draw=none, inner sep = 0] at (-1.6,-1.6) {\Large{$y \otr x$}};
\node [draw=none, inner sep = 0] at (1.5,-1.6) {\Large{$x $}};
\end{tikzpicture}
\hspace{4cm}
&
\begin{tikzpicture}[domain=-2:2, scale=1, knot gap=9, transform canvas={scale=0.65}]
\draw [knot=black,line width=0.50mm,
decoration={markings, mark=at position 0.25 with {\arrow[black,ultra thick]{>}}},
decoration={markings, mark=at position 0.76 with {\arrow[black,ultra thick]{>}}}, postaction={decorate}]
(-1.5,-1.3229) to  [out=-135, in=45, looseness=0] (1.5,1.3229) ;
\draw [knot=black,line width=0.50mm,
decoration={markings, mark=at position 0.25  with {\arrow[black,ultra thick]{>}}},
decoration={markings, mark=at position 0.76 with {\arrow[black,ultra thick]{>}}}, postaction={decorate}]
(-1.5,1.3229) to [out=-45, in=135, looseness=0] (1.5,-1.3229);
\node [draw=none, inner sep = 0] at (-1.6,1.6) {\Large{$x$}};
\node [draw=none, inner sep = 0] at (1.6,1.6) {\Large{$y \tr x$}};
\node [draw=none, inner sep = 0] at (-1.6,-1.6) {\Large{$y$}};
\node [draw=none, inner sep = 0] at (1.6,-1.6) {\Large{$x$}};
\end{tikzpicture}
\hspace{4cm}
&
\begin{tikzpicture}[domain=-2:2, scale=1, knot gap=9, transform canvas={scale=0.65}]
\draw [knot=black,line width=0.50mm,
decoration={markings, mark=at position 0.25 with {\arrow[black,ultra thick]{<}}},
decoration={markings, mark=at position 0.75 with {\arrow[black,ultra thick]{<}}}, postaction={decorate}]
(-1,1.3229) to [out=right, in=left, looseness=1] (0.5,0) to [out=left, in=right, looseness=1] (-1,-1.3229);
\node [draw=none, inner sep = 0] at (-1.1,1.7) {\Large{$f(x)$}};
\node [draw=none, inner sep = 0] at (-1,-1.6) {\Large{$x$}};
\end{tikzpicture}
\end{tabular}
\]
\vspace{1cm}

The definition of \textit{Legendrian rack} is motivated by the diagrams of Legendrian Reidemeister moves subject to the above relation (see the figures below). The type I move has four diagrams, but we include only two diagrams. It is easy to check that the other two diagrams do not give different relations. 




\vspace{1cm}
\[
\begin{tabular}{lll}
\begin{tikzpicture}[domain=-2:2, scale=1, knot gap=9, transform canvas={scale=0.65}]
\draw [knot=black,line width=0.50mm,
decoration={markings, mark=at position 0.25 with {\arrow[black,ultra thick]{<}}},
decoration={markings, mark=at position 0.75 with {\arrow[black,ultra thick]{<}}}, postaction={decorate}]
(1.732,1) to  [out=left, in=45, looseness=0.7] (-1.732,-1);
\draw [knot=black,line width=0.50mm,
decoration={markings, mark=at position 0.25 with {\arrow[black,ultra thick]{>}}},
decoration={markings, mark=at position 0.75 with {\arrow[black,ultra thick]{>}}}, postaction={decorate}]
(-1.732,1) to  [out=right, in=135, looseness=0.7] (1.732,-1) ;
\draw [knot=black, knot gap=0,line width=0.50mm,
decoration={markings, mark=at position 0.25  with {\arrow[black,ultra thick]{<}}},
decoration={markings, mark=at position 0.75 with {\arrow[black,ultra thick]{<}}}, postaction={decorate}]
(-1.732,1) to [out=right, in=left, looseness=0.7] (0,1.5) to [out=right, in=left, looseness=0.7] (1.732,1);
\node [draw=none, inner sep = 0] at (-1.7,0.5) {\Large{$f^2(x \tr x)$}};
\node [draw=none, inner sep = 0] at (1.5,0.5) {\Large{$x \tr x$}};
\node [draw=none, inner sep = 0] at (1.2,1.8) {\Large{$f(x \tr x)$}};
\node [draw=none, inner sep = 0] at (-1.8,-1.2) {\Large{$x$}};
\node [draw=none, inner sep = 0] at (1.8,-1.2) {\Large{$x$}};
\end{tikzpicture}
\hspace{2cm}
&
\begin{tikzpicture}[scale=.9, transform shape]
\draw[>=triangle 45, <->] (0,0) -- (2,0);
\end{tikzpicture}
\hspace{2cm}
&
\begin{tikzpicture}[domain=-2:2, scale=1, knot gap=9,transform canvas={scale=0.65}]
\draw [knot=black,line width=0.50mm,
decoration={markings, mark=at position 0.25  with {\arrow[black,ultra thick]{>}}},
decoration={markings, mark=at position 0.75 with {\arrow[black,ultra thick]{>}}}, postaction={decorate}]
(-2,0) to [out=45, in=left, looseness=1] (0,1) to [out=right, in=135, looseness=1] (2,0);
\node [draw=none, inner sep = 0] at (-2,-0.2) {\Large{$x$}};
\node [draw=none, inner sep = 0] at (2,-0.2) {\Large{$x$}};
\end{tikzpicture}
\end{tabular}
\]
\vspace{2cm}
\[
\begin{tabular}{lll}
\begin{tikzpicture}[domain=-2:2, scale=1, knot gap=9, transform canvas={scale=0.65}]
\draw [knot=black,line width=0.50mm,
decoration={markings, mark=at position 0.25 with {\arrow[black,ultra thick]{>}}},
decoration={markings, mark=at position 0.75 with {\arrow[black,ultra thick]{>}}}, postaction={decorate}]
(1.732,1) to  [out=left, in=45, looseness=0.7] (-1.732,-1);
\draw [knot=black,line width=0.50mm,
decoration={markings, mark=at position 0.25 with {\arrow[black,ultra thick]{<}}},
decoration={markings, mark=at position 0.75 with {\arrow[black,ultra thick]{<}}}, postaction={decorate}]
(-1.732,1) to  [out=right, in=135, looseness=0.7] (1.732,-1) ;
\draw [knot=black, knot gap=0,line width=0.50mm,
decoration={markings, mark=at position 0.25  with {\arrow[black,ultra thick]{>}}},
decoration={markings, mark=at position 0.75 with {\arrow[black,ultra thick]{>}}}, postaction={decorate}]
(-1.732,1) to [out=right, in=left, looseness=0.7] (0,1.5) to [out=right, in=left, looseness=0.7] (1.732,1);
\node [draw=none, inner sep = 0] at (1.3,0.4) {\Large{$f^2(x)$}};
\node [draw=none, inner sep = 0] at (1.3,1.8) {\Large{$f(x)$}};
\node [draw=none, inner sep = 0] at (-1.8,-1.2) {\Large{$f^2(x) \tr x$}};
\node [draw=none, inner sep = 0] at (1.8,-1.2) {\Large{$x$}};
\end{tikzpicture}
\hspace{2cm}
&
\begin{tikzpicture}[scale=.9, transform shape]
\draw[>=triangle 45, <->] (0,0) -- (2,0);
\end{tikzpicture}
\hspace{2cm}
&
\begin{tikzpicture}[domain=-2:2, scale=1, knot gap=9,transform canvas={scale=0.65}]
\draw [knot=black,line width=0.50mm,
decoration={markings, mark=at position 0.25  with {\arrow[black,ultra thick]{<}}},
decoration={markings, mark=at position 0.75 with {\arrow[black,ultra thick]{<}}}, postaction={decorate}]
(-2,0) to [out=45, in=left, looseness=1] (0,1) to [out=right, in=135, looseness=1] (2,0);
\node [draw=none, inner sep = 0] at (-2,-0.2) {\Large{$x$}};
\node [draw=none, inner sep = 0] at (2,-0.2) {\Large{$x$}};
\end{tikzpicture}
\end{tabular}
\]
\vspace{1cm}

Now we consider the four diagrams coming from the Legendrian Reidemeister move type II. 



\vspace{1cm}

\[
\begin{tabular}{lll}
\begin{tikzpicture}[domain=-2:2, scale=1, knot gap=9, transform canvas={scale=0.65}]
\draw [knot=black,line width=0.50mm,
decoration={markings, mark=at position 0.20  with {\arrow[black,ultra thick]{<}}},
decoration={markings, mark=at position 0.80 with {\arrow[black,ultra thick]{<}}}, postaction={decorate}]
(-1.5,1.5) to [out=right, in=left, looseness=1] (2,0) to [out=left, in=right, looseness=1] (-1.5,-1.5);
\draw [knot=black,line width=0.50mm,
decoration={markings, mark=at position 0.25 with {\arrow[black,ultra thick]{>}}},
decoration={markings, mark=at position 0.50 with {\arrow[black,ultra thick]{>}}},
decoration={markings, mark=at position 0.76 with {\arrow[black,ultra thick]{>}}}, postaction={decorate}]
(0.2,1.9365) to [out=down, in=up, looseness=0] (1,-1.9365) ;
\node [draw=none, inner sep = 0] at (-1.8,2) {\Large{$f(x \tr y) \otr y$}};
\node [draw=none, inner sep = 0] at (0.2,2.2) {\Large{$y$}};
\node [draw=none, inner sep = 0] at (1.55,0.80) {\Large{$f(x \tr y)$}};
\node [draw=none, inner sep = 0] at (1.55,-0.60) {\Large{$x \tr y$}};
\node [draw=none, inner sep = 0] at (-1.5,-2) {\Large{$x$}};
\node [draw=none, inner sep = 0] at (1,-2.2) {\Large{$y$}};
\end{tikzpicture}
\hspace{2cm}
&
\begin{tikzpicture}[scale=.9, transform shape]
\draw[>=triangle 45, <->] (0,0) -- (2,0);
\end{tikzpicture}
\hspace{2cm}
&
\begin{tikzpicture}[domain=-2:2, scale=1, knot gap=9, transform canvas={scale=0.65}]
\draw [knot=black,line width=0.50mm,
decoration={markings, mark=at position 0.25 with {\arrow[black,ultra thick]{<}}},
decoration={markings, mark=at position 0.75 with {\arrow[black,ultra thick]{<}}}, postaction={decorate}]
(-2,1.5) to [out=right, in=left, looseness=1] (0.5,0) to [out=left, in=right, looseness=1] (-2,-1.5);
 \draw [knot=black,line width=0.50mm,
decoration={markings, mark=at position 0.25 with {\arrow[black,ultra thick]{>}}},
decoration={markings, mark=at position 0.76 with {\arrow[black,ultra thick]{>}}}, postaction={decorate}]
(0.5,1.9365) to [out=down, in=up, looseness=0] (1.5,-1.9365) ;
\node [draw=none, inner sep = 0] at (-2,2) {\Large{$f(x)$}};
\node [draw=none, inner sep = 0] at (.5,2.2) {\Large{$y$}};
\node [draw=none, inner sep = 0] at (1.5,-2.2) {\Large{$y$}};
\node [draw=none, inner sep = 0] at (-2,-2) {\Large{$x$}};
\end{tikzpicture}
\end{tabular}
\]
\vspace{3cm}
\[
\begin{tabular}{lll}
\begin{tikzpicture}[domain=-2:2, scale=1, knot gap=9, transform canvas={scale=0.65}]
\draw [knot=black,line width=0.50mm,
decoration={markings, mark=at position 0.20  with {\arrow[black,ultra thick]{<}}},
decoration={markings, mark=at position 0.80 with {\arrow[black,ultra thick]{<}}}, postaction={decorate}]
(-1.5,1.5) to [out=right, in=left, looseness=1] (2,0) to [out=left, in=right, looseness=1] (-1.5,-1.5);
\draw [knot=black,line width=0.50mm,
decoration={markings, mark=at position 0.25 with {\arrow[black,ultra thick]{<}}},
decoration={markings, mark=at position 0.50 with {\arrow[black,ultra thick]{<}}},
decoration={markings, mark=at position 0.76 with {\arrow[black,ultra thick]{<}}}, postaction={decorate}]
(0.2,1.9365) to [out=down, in=up, looseness=0] (1,-1.9365) ;
\node [draw=none, inner sep = 0] at (-1.8,2) {\Large{$f(x \otr y) \tr y$}};
\node [draw=none, inner sep = 0] at (0.2,2.2) {\Large{$y$}};
\node [draw=none, inner sep = 0] at (1.55,0.80) {\Large{$f(x \otr y)$}};
\node [draw=none, inner sep = 0] at (1.55,-0.60) {\Large{$x \otr y$}};
\node [draw=none, inner sep = 0] at (-1.5,-2) {\Large{$x$}};
\node [draw=none, inner sep = 0] at (1,-2.2) {\Large{$y$}};
\end{tikzpicture}
\hspace{2cm}
&
\begin{tikzpicture}[scale=.9, transform shape]
\draw[>=triangle 45, <->] (0,0) -- (2,0);
\end{tikzpicture}
\hspace{2cm}
&
\begin{tikzpicture}[domain=-2:2, scale=1, knot gap=9, transform canvas={scale=0.65}]
\draw [knot=black,line width=0.50mm,
decoration={markings, mark=at position 0.25 with {\arrow[black,ultra thick]{<}}},
decoration={markings, mark=at position 0.75 with {\arrow[black,ultra thick]{<}}}, postaction={decorate}]
(-2,1.5) to [out=right, in=left, looseness=1] (0.5,0) to [out=left, in=right, looseness=1] (-2,-1.5);
 \draw [knot=black,line width=0.50mm,
decoration={markings, mark=at position 0.25 with {\arrow[black,ultra thick]{<}}},
decoration={markings, mark=at position 0.76 with {\arrow[black,ultra thick]{<}}}, postaction={decorate}]
(0.5,1.9365) to [out=down, in=up, looseness=0] (1.5,-1.9365) ;
\node [draw=none, inner sep = 0] at (-2,2) {\Large{$f(x)$}};
\node [draw=none, inner sep = 0] at (.5,2.2) {\Large{$y$}};
\node [draw=none, inner sep = 0] at (1.5,-2.2) {\Large{$y$}};
\node [draw=none, inner sep = 0] at (-2,-2) {\Large{$x$}};
\end{tikzpicture}
\end{tabular}
\]
\vspace{3cm}
\[
\begin{tabular}{lll}
\begin{tikzpicture}[domain=-2:2, scale=1, knot gap=9, transform canvas={scale=0.65}]
\draw [knot=black,line width=0.50mm,
decoration={markings, mark=at position 0.25 with {\arrow[black,ultra thick]{<}}},
decoration={markings, mark=at position 0.50 with {\arrow[black,ultra thick]{<}}},
decoration={markings, mark=at position 0.76 with {\arrow[black,ultra thick]{<}}}, postaction={decorate}]
(1.5,1.9365) to [out=down, in=up, looseness=0] (0,-1.9365) ;
\draw [knot=black,line width=0.50mm,
decoration={markings, mark=at position 0.20  with {\arrow[black,ultra thick]{<}}},
decoration={markings, mark=at position 0.80 with {\arrow[black,ultra thick]{<}}}, postaction={decorate}]
(2,1) to [out=left, in=right, looseness=1] (.7, 1.5) to [out=left, in=right, looseness=1] (-2,0) to [out=right, in=left, looseness=1] (2,-1.5);
\node [draw=none, inner sep = 0] at (-.8,2.2) {\Large{$f(x)$}};
\node [draw=none, inner sep = 0] at (2,2) {\Large{$(y \otr x) \tr f(x)$}};
\node [draw=none, inner sep = 0] at (1.5,0) {\Large{$y \otr x$}};
\node [draw=none, inner sep = 0] at (-.5,-2.2) {\Large{$y$}};
\node [draw=none, inner sep = 0] at (2,-2) {\Large{$x$}};
\end{tikzpicture}
\hspace{2cm}
&
\begin{tikzpicture}[scale=.9, transform shape]
\draw[>=triangle 45, <->] (0,0) -- (2,0);
\end{tikzpicture}
\hspace{2cm}
&
\begin{tikzpicture}[domain=-2:2, scale=1, knot gap=9, transform canvas={scale=0.65}]
\draw [knot=black,line width=0.50mm,
decoration={markings, mark=at position 0.20  with {\arrow[black,ultra thick]{<}}},
decoration={markings, mark=at position 0.80 with {\arrow[black,ultra thick]{<}}}, postaction={decorate}]
(2,1.5) to [out=left, in=right, looseness=1] (-.5,0) to [out=right, in=left, looseness=1] (2,-1.5);
 \draw [knot=black,line width=0.50mm,
decoration={markings, mark=at position 0.25 with {\arrow[black,ultra thick]{<}}},
decoration={markings, mark=at position 0.76 with {\arrow[black,ultra thick]{<}}}, postaction={decorate}]
(0,1.9365) to [out=down, in=up, looseness=0] (-1.5,-1.9365) ;
\node [draw=none, inner sep = 0] at (2,2) {\Large{$f(x)$}};
\node [draw=none, inner sep = 0] at (0,2.2) {\Large{$y$}};
\node [draw=none, inner sep = 0] at (2,-2) {\Large{$x$}};
\node [draw=none, inner sep = 0] at (-1.5,-2.2) {\Large{$y$}};
\end{tikzpicture}
\end{tabular}
\]
\vspace{3cm}
\[
\begin{tabular}{lll}
\begin{tikzpicture}[domain=-2:2, scale=1, knot gap=9, transform canvas={scale=0.65}]
\draw [knot=black,line width=0.50mm,
decoration={markings, mark=at position 0.25 with {\arrow[black,ultra thick]{>}}},
decoration={markings, mark=at position 0.50 with {\arrow[black,ultra thick]{>}}},
decoration={markings, mark=at position 0.76 with {\arrow[black,ultra thick]{>}}}, postaction={decorate}]
(1.5,1.9365) to [out=down, in=up, looseness=0] (0,-1.9365) ;
\draw [knot=black,line width=0.50mm,
decoration={markings, mark=at position 0.20  with {\arrow[black,ultra thick]{<}}},
decoration={markings, mark=at position 0.80 with {\arrow[black,ultra thick]{<}}}, postaction={decorate}]
(2,1) to [out=left, in=right, looseness=1] (.7, 1.5) to [out=left, in=right, looseness=1] (-2,0) to [out=right, in=left, looseness=1] (2,-1.5);
\node [draw=none, inner sep = 0] at (0,2.2) {\Large{$f(x)$}};
\node [draw=none, inner sep = 0] at (2,2) {\Large{$y$}};
\node [draw=none, inner sep = 0] at (1.8,0) {\Large{$y \otr f(x)$}};
\node [draw=none, inner sep = 0] at (-1.5,-2.2) {\Large{$(y \otr f(x)) \tr x$}};
\node [draw=none, inner sep = 0] at (2,-2) {\Large{$x$}};
\end{tikzpicture}
\hspace{2cm}
&
\begin{tikzpicture}[scale=.9, transform shape]
\draw[>=triangle 45, <->] (0,0) -- (2,0);
\end{tikzpicture}
\hspace{2cm}
&
\begin{tikzpicture}[domain=-2:2, scale=1, knot gap=9, transform canvas={scale=0.65}]
\draw [knot=black,line width=0.50mm,
decoration={markings, mark=at position 0.20  with {\arrow[black,ultra thick]{<}}},
decoration={markings, mark=at position 0.80 with {\arrow[black,ultra thick]{<}}}, postaction={decorate}]
(2,1.5) to [out=left, in=right, looseness=1] (-.5,0) to [out=right, in=left, looseness=1] (2,-1.5);
 \draw [knot=black,line width=0.50mm,
decoration={markings, mark=at position 0.25 with {\arrow[black,ultra thick]{>}}},
decoration={markings, mark=at position 0.76 with {\arrow[black,ultra thick]{>}}}, postaction={decorate}]
(0,1.9365) to [out=down, in=up, looseness=0] (-1.5,-1.9365) ;
\node [draw=none, inner sep = 0] at (2,2) {\Large{$f(x)$}};
\node [draw=none, inner sep = 0] at (0,2.2) {\Large{$y$}};
\node [draw=none, inner sep = 0] at (2,-2) {\Large{$x$}};
\node [draw=none, inner sep = 0] at (-1.5,-2.2) {\Large{$y$}};
\end{tikzpicture}
\end{tabular}
\]
\vspace{2cm}

Lastly, we consider the type III Legendrian Reidemeister move. 




\vspace{1cm}
\[
\begin{tabular}{lll}
\begin{tikzpicture}[domain=-2:2, scale=1, knot gap=9, transform canvas={scale=0.65}]
\draw [knot=black,line width=0.50mm, 
decoration={markings, mark=at position 0.25  with {\arrow[black,ultra thick]{<}}},
decoration={markings, mark=at position 0.76 with {\arrow[black,ultra thick]{<}}}, postaction={decorate}]
(2,2) to  [out=down, in=up, looseness=0] (-2,-2);
\draw [knot=black,line width=0.50mm, 
decoration={markings, mark=at position 0.20  with {\arrow[black,ultra thick]{<}}},
decoration={markings, mark=at position 0.50  with {\arrow[black,ultra thick]{<}}},
decoration={markings, mark=at position 0.80 with {\arrow[black,ultra thick]{<}}}, postaction={decorate}]
(2,0) to  [out=left, in=right, looseness=1] (0,1) to  [out=left, in=right, looseness=1] (-2,0)  ;
\draw [knot=black,line width=0.50mm, 
decoration={markings, mark=at position 0.25  with {\arrow[black,ultra thick]{>}}},
decoration={markings, mark=at position 0.76 with {\arrow[black,ultra thick]{>}}}, postaction={decorate}]
(-2,2) to  [out=down, in=up, looseness=0] (2,-2);
\node [draw=none, inner sep = 0] at (-2,2.2) {\Large{$z$}};
\node [draw=none, inner sep = 0] at (2,-2.2) {\Large{$z$}};
\node [draw=none, inner sep = 0] at (2,-.3) {\Large{$y \tr z$}};
\node [draw=none, inner sep = 0] at (-2,-.3) {\Large{$y$}};
\node [draw=none, inner sep = 0] at (-2,-2.2) {\Large{$x$}};
\node [draw=none, inner sep = 0] at (.7,0) {\large{$x \tr z$}};
\node [draw=none, inner sep = 0] at (2,2.3) {\Large{$(x \tr z) \tr (y \tr z)$}};
\end{tikzpicture}
\hspace{2cm}
&
\begin{tikzpicture}[scale=.9, transform shape]
\draw[>=triangle 45, <->] (0,0) -- (2,0);
\end{tikzpicture}
\hspace{2cm}
&
\begin{tikzpicture}[domain=-2:2, scale=1, knot gap=9, transform canvas={scale=0.65}]
\draw [knot=black,line width=0.50mm, 
decoration={markings, mark=at position 0.25  with {\arrow[black,ultra thick]{<}}},
decoration={markings, mark=at position 0.76 with {\arrow[black,ultra thick]{<}}}, postaction={decorate}]
(2,2) to  [out=down, in=up, looseness=0] (-2,-2);
\draw [knot=black,line width=0.50mm, 
decoration={markings, mark=at position 0.20  with {\arrow[black,ultra thick]{<}}},
decoration={markings, mark=at position 0.50  with {\arrow[black,ultra thick]{<}}},
decoration={markings, mark=at position 0.80 with {\arrow[black,ultra thick]{<}}}, postaction={decorate}]
(2,0) to  [out=left, in=right, looseness=1] (0,-1) to  [out=left, in=right, looseness=1] (-2,0)  ;
\draw [knot=black,line width=0.50mm, 
decoration={markings, mark=at position 0.25  with {\arrow[black,ultra thick]{>}}},
decoration={markings, mark=at position 0.76 with {\arrow[black,ultra thick]{>}}}, postaction={decorate}]
(-2,2) to  [out=down, in=up, looseness=0] (2,-2);
\node [draw=none, inner sep = 0] at (-2,2.2) {\Large{$z$}};
\node [draw=none, inner sep = 0] at (2,-2.2) {\Large{$z$}};
\node [draw=none, inner sep = 0] at (2,-.3) {\Large{$y \tr z$}};
\node [draw=none, inner sep = 0] at (-2,-.3) {\Large{$y$}};
\node [draw=none, inner sep = 0] at (-2,-2.2) {\Large{$x$}};
\node [draw=none, inner sep = 0] at (-.7,-.1) {\large{$x \tr y$}};
\node [draw=none, inner sep = 0] at (2,2.3) {\Large{$(x \tr y) \tr z$}};
\end{tikzpicture}
\end{tabular}
\]
\vspace{2cm}


Thus we can make the following definition:

\begin{definition}\label{LegendleDef}
	A \textit{Legendrian rack}  is a triple $(X, \tr, f)$, where 
$(X, \tr)$ is a rack and  $f: X \rightarrow X$ is a map such that the 
following properties hold for all $x, y \in X$:
	\begin{list}{}{} 
		\item[\textup{(I)}]{$ f^2(x \tr x) =x=f^2(x) \tr x$},
		\item[\textup{(II)}]{$ f(x \tr y)= f(x) \tr y $} 
		\item[\textup{(III)}]{$x  \tr f(y)= x \tr y$}. 
	\end{list}	
The map $f$ is called a \textit{Legendrian map} or 
\textit{Legendrian structure} on $X$.
\end{definition}

By construction, we have the following:

\begin{proposition}
Let $(X,\tr,f)$ be a Legendrian rack. Then the number $\Phi_X^{\mathbb{Z}}(L)$ of colorings of a front projection $L$ of a Legendrian knot or link is an integer-valued invariant of Legendrian isotopy. We call this number of colorings the \textit{Legendrian rack counting invariant}.
\end{proposition}

\begin{remark}
Note that if the rack operation $\tr$ is idempotent, then the map $f$ in 
Definition~\ref{LegendleDef} becomes an involution.
\end{remark}

\begin{proposition}
If $(X, \tr, f)$ is finite Legendrian rack, then the map $f$ is an automorphism of the rack $(X, \tr)$.
\end{proposition}

\begin{proof} Let $(X, \tr, f)$ be a Legendrian rack.  Then the 
conditions (II) and (III) of Definition~\ref{LegendleDef} imply that 
\[
f(x \tr y)=f(x  \tr f(y))=f(x) \tr f(y),
\]
making $f$ a homomorphism of the rack $(X, \tr)$.
Now if $f(x)=f(y)$ then we have
\[
x=f^2(x)\tr x=f^2(y) \tr x=f^2(y) \tr f^2(x)=f^2(y) \tr f^2(y)= f^2(y) \tr y=y,
\]
giving bijectivity since $X$ is finite set.  Thus the map $f$ is a rack 
automorphism.
\end{proof}

\begin{remark}
	Notice that the converse of this proposition is not true.  Take $X=\mathbb{Z}_4$ with $x \tr y=x+1$ and $f(x)=x+1.$  The first condition of Definition~\ref{LegendleDef} is not satisfied since $f^2(x \tr x)=f^2(x+1)=x+3\neq x.$ 
\end{remark}

Automorphisms of quandles and racks have been investigated in \cite{EMR}, where it was shown that automorphism of dihedral quandles are affine maps $f(x)=ax+b$.

\begin{definition}
Let $(X, \tr_X, f_X)$ and $(Y,\tr_Y, f_Y)$ be two Legendrian racks.  
A {\it Legendrian rack homomorphism} between $(X, \tr_X, f_X)$ and 
$(Y,\tr_Y, f_Y)$ is a rack homomorphism 
$\psi: (X, \tr_X) \rightarrow (Y,\tr_Y)$ such that 
$f_Y \circ \psi=\psi \circ f_X $, where
 $\tr_X$ and $\tr_Y$ 
 denote the rack operations of $X$ and $Y$, respectively.
 A {\it Legendrian rack isomorphism} is a bijective Legendrian rack 
homomorphism, and two Legendrian racks are {\it isomorphic} if there is a 
Legendrian rack isomorphism between them.  
\end{definition}

Let $X$ be a $(t,s)$-rack and consider a map $f:X\to X$ defined
by $f(x)=ax+b$ for some $a,b\in X$. What conditions are needed to make $f$
a Legendrian structure?

Condition (I) says that
\[a^2(t+s)x+(ab+b)=x=(a^2t+s)x+t(ab+b)\]
which implies $ab+b=(a+1)b=0$ and
$a^2(t+s)=1=a^2t+s$. Then $a^2s=s$ implies $(1-a^2)s=0$, so we obtain the
necessary and jointly sufficient conditions $a^2(t+s)=1$ and $(1-a^2)s=0$
for (I).


Condition (II) says that 
\[a(tx+sy)+b = 
atx+asy+b=
t(ax+b)+sy=
atx+tb+sy
\]
so we must have $(1-a)s=0$ and $(1-t)b=0$, and
condition (III) says
\[tx+s(ay+b)=tx+asy+sb=tx+sy\]
so we must have $sb=0$ and $(1-a)s=0$. Collecting the conditions together, 
we have proved:

\begin{proposition}\label{ts-Leg}
Let $X$ be a $(t,s)$-rack, i.e. a $\mathbb{Z}[t^{\pm 1},s]/(s^2-(1-t)s)$-module. 
Then $X$ is a Legendrian rack under the operations
\[x\tr y = tx+sy \quad \mathrm{and}\quad f(x)=ax+b\]
for $a,b\in X$ if and only if
$a^2(t+s)=1$ and  $(a+1)b=(1-a)s=(1-t)b=sb=0$.
\end{proposition}

\begin{example}\label{Z8}
Consider $\mathbb{Z}_{8}$ as a rack with operation
	\[
	x\tr y=3x-2y. 
	\]
	Then the map $f:\mathbb{Z}_{8} \rightarrow \mathbb{Z}_{8}$ 
given by $f(x)=a x+b$ is a Legendrian map for $(a,b)\in\{(1,0), (1,4), (5,0), (5,4)\}$.
\end{example}

\begin{example}
Consider the Legendrian rack $(\mathbb{Z}_{8}, \tr, f)$  with $x\tr y = 3x-2y$ and
$f(x)=5x+4$.  Any map 
$\psi:\mathbb{Z}_{8} \rightarrow \mathbb{Z}_{8}$ given by $\psi(x)=ax+a-1$, 
where $a \in \mathbb{Z}_{8}$, gives a Legendrian rack endomorphism.  
Furthermore, if $a$ is invertible in $\mathbb{Z}_{8}$, then $\psi$ is an 
automorphism.
\end{example}

\begin{example}\label{Z10}
Consider $\mathbb{Z}_{10}$ as a rack with operation
	\[x\tr y=3x-2y 	\]
Since the only square roots of $1$ are $1$ and $9$, the condition 
$(1-a)s=0$ is only satisfied for $a=1$; then $sb=0$ requires $b=5$, and 
we check that $(a+1)b=(1+1)5=0$, $(1-a)s=(1-1)(-2)=0$ and $(1-t)b=(1-3)5=(-2)5=0$ so the only
Legendrian map of the form $f(x)=ax+b$ on this rack is $f(x)=x+5$.
	\end{example}

We can define racks and quandles without algebraic formulas by listing their 
operation tables in the from of a matrix. Specifically, we can specify an
operation $\triangleright$ on the set $\{1,2,\dots, n\}$ with a matrix
$M$ whose entry in row $j$ column $k$ is $j\triangleright k$. 

\begin{example}\label{3elements}
Up to isomorphism, there are six racks of three elements.
For each of these racks, we list the possible Legendrian
maps $f\in S_3$ in cycle notation in the table below.
\[\begin{array}{r|l}
M & f\\ \hline
& \\
\left[\begin{array}{rrr}
1 & 1 & 1 \\
2 & 2 & 2 \\
3 & 3 & 3 
\end{array}\right] & 
f=(),(12), (13), (23) \\
& \\
\left[\begin{array}{rrr}
1 & 1 & 1 \\
3 & 2 & 2 \\
2 & 3 & 3 
\end{array}\right] & 
f=(), (23) \\
& \\
\left[\begin{array}{rrr}
1 & 3 & 2 \\
3 & 2 & 1 \\
2 & 1 & 3 
\end{array}\right] & 
f=() 
\end{array}\quad \begin{array}{r|l} M & f\\ \hline & \\
\left[\begin{array}{rrr}
2 & 2 & 2 \\
3 & 3 & 3 \\
1 & 1 & 1 
\end{array}\right] & 
(123) \\
& \\
\left[\begin{array}{rrr}
2 & 2 & 2 \\
1 & 1 & 1 \\
3 & 3 & 3 \\
\end{array}\right] & 
- \\
& \\
\left[\begin{array}{rrr}
2 & 2 & 1 \\
1 & 1 & 2 \\
3 & 3 & 3 \\
\end{array}\right] & 
- \\
\end{array}
\]
\end{example}

In the previous examples, we note that only quandles seem to have
Legendrian maps. Our next example shows that non-quandle racks
can have Legendrian maps.

\begin{example}\label{nonq}
The constant action rack structure on the set $\{1,2,3,4\}$ given by 
$x\tr y =\sigma(x)$, where (in cycle notation) $\sigma=(12)(34)$, 
has two Legendrian maps,
$f_1=(1324)$ and $f_2=(1423)$ as can be verified easily.
For example, $f_1^2=f_2^2=(12)(34)=\sigma$, so axiom (I) becomes $x=\sigma^2(x)$, axiom (II) becomes $f_i\sigma= \sigma f_i$ for $(i=1,2)$ and axiom (III) becomes a tautology.
\end{example}

The following example of a rack satisfies the conditions of Proposition~\ref{ts-Leg} where the map $f$ is an involution.    

\begin{example}
	Consider $\mathbb{Z}_{4}$ as a rack with operation
	\[x\tr y=3x+2y. 	\]
	The map $f(x)=-x$ makes this rack into a Legendrian rack.	
\end{example}

The following example of a rack that is not a quandle satisfies the conditions of Proposition~\ref{ts-Leg} where the map $f$ is not an involution.    

\begin{example}
	Consider $\mathbb{Z}_{49}$ as a rack with operation
	\[x\tr y=2x. 	\]
	The map $f(x)=5x$ makes this rack into a Legendrian rack.	
\end{example}

\begin{example}
Of the 19 isomorphism classes of racks with four elements, we find that 11
have nonempty sets of Legendrian structures. These are:
\[\begin{array}{r|l r|l}
M & f & M & f\\ \hline
& & & \\
\left[\begin{array}{rrrr}
1 & 3 & 4 & 2\\
4 & 2 & 1 & 3 \\
2 & 4 & 3 & 1 \\
3 & 1 & 2 & 4
\end{array}\right] 
& () 
& 
\left[\begin{array}{rrrr}
2 & 2 & 1 & 2 \\
4 & 4 & 2 & 4 \\
3 & 3 & 3 & 3 \\
1 & 1 & 4 & 1
\end{array}\right] & (124) \\
& & & \\
\left[\begin{array}{rrrr}
2 & 2 & 2 & 2 \\
3 & 3 & 3 & 3 \\
1 & 1 & 1 & 1 \\
4 & 4 & 4 & 4
\end{array}\right] 
& (123) 
& 
\left[\begin{array}{rrrr}
2 & 2 & 2 & 3 \\
3 & 3 & 3 & 1 \\
1 & 1 & 1 & 2 \\
4 & 4 & 4 & 4
\end{array}\right] 
& (123)\\
& & & \\
\left[\begin{array}{rrrr}
1 & 1 & 4 & 3 \\
2 & 2 & 2 & 2 \\
4 & 3 & 3 & 1 \\
3 & 4 & 1 & 4
\end{array}\right] 
& ()  
& 
\left[\begin{array}{rrrr}
1 & 3 & 1 & 1 \\
2 & 2 & 2 & 2 \\
3 & 4 & 3 & 3 \\
4 & 1 & 4 & 4
\end{array}\right] 
& () \\
& & & \\
\left[\begin{array}{rrrr}
1 & 4 & 4 & 1 \\
3 & 2 & 2 & 3 \\
2 & 3 & 3 & 2 \\
4 & 1 & 1 & 4
\end{array}\right] 
& (23),(14),(14)(23) 
& 
\left[\begin{array}{rrrr}
2 & 2 & 2 & 2 \\
1 & 1 & 1 & 1 \\
4 & 4 & 4 & 4 \\
3 & 3 & 3 & 3
\end{array}\right] 
& (1324),(1423) \\
& & & \\
\left[\begin{array}{rrrr}
1 & 1 & 4 & 1 \\
2 & 2 & 2 & 2 \\
3 & 3 & 3 & 3 \\
4 & 4 & 1 & 4
\end{array}\right] 
& (),(14) 
&
\left[\begin{array}{rrrr}
1 & 3 & 1 & 3 \\
2 & 2 & 2 & 2 \\
3 & 1 & 3 & 1 \\
4 & 4 & 4 & 4
\end{array}\right] 
& (13),(24),(13)(24),() \\
& & & \\
\left[\begin{array}{rrrr}
1 & 1 & 1 & 1 \\
2 & 2 & 2 & 2 \\
3 & 3 & 3 & 3 \\
4 & 4 & 4 & 4
\end{array}\right] 
& 
\begin{array}{l}
(12)(34),(13)(24),(14)(23),\\
(34),(23),(24),(12),(13),(14),().
\end{array}
 \\
\end{array}
\]
\end{example}

\section{Distinguishing Legendrian Knots using Legendrian Racks}\label{Col}

In this section we use coloring of Legendrian knot diagrams to distinguish some Legendrian knots. In the first three examples we respectively distinguish between the unknot and its positive stabilization, the trefoil and its negative stabilization and also the trefoil and its positive stabilization. The last two examples deal with distinguishing connected sum of Legendrian knots and distinguishing the two Legendrian knots of topological type $6_2$.

Note that crossing information is not denoted in the following diagrams since in a front projection of a Legendrian knot only contains crossings were the overstrand has a smaller slope than the understrand:
\vspace{.5cm}
\[\begin{tikzpicture}[domain=-2:2, scale=1, knot gap=9, transform canvas={scale=0.65}]
\draw [knot=black,line width=0.50mm]
(1.5,1.3229) to  [out=-135, in=45, looseness=0] (-1.5,-1.3229) ;
\draw [knot=black,line width=0.50mm]
(-1.5,1.3229) to [out=-45, in=135, looseness=0] (1.5,-1.3229);
\end{tikzpicture}
  \]
  \vspace{.5cm}

 Now we start with the following example distinguishing between the unknot and its positive stabilization.

\begin{example}\label{UnknotStab}
Consider the following diagrams of the unknot and its positive stabilization.
A coloring of the diagram of the unknot  by $(X, \tr, f)$ gives the condition $f^2(x)=x$, while a coloring of the diagram of its positive stabilization by $(X, \tr, f)$ gives the condition $f^4(x)=x$.  Now by choosing $(X, \tr, f)$ to be the Legendrian rack given in Example~\ref{nonq} and since $f^4$ is the identity map while $f^2$ is not, the two knots are thus distinguished by their sets of colorings. 

\vspace{.7cm}
\[
\begin{tabular}{ll}
\begin{tikzpicture}[domain=-2:2, scale=1, knot gap=9, transform canvas={scale=0.65}]
\draw [knot=black, knot gap=0,line width=0.50mm,
decoration={markings, mark=at position 0.50  with {\arrow[black,ultra thick]{<}}},
postaction={decorate}](-2,0) to [out=right, in=left, looseness=0.7] (0,1) to [out=right, in=left, looseness=0.7] (2,0);
\draw [knot=black, knot gap=0,line width=0.50mm,decoration={markings, mark=at position 0.50  with {\arrow[black,ultra thick]{>}}},
postaction={decorate}](-2,0) to [out=right, in=left, looseness=0.7] (0,-1) to [out=right, in=left, looseness=0.7] (2,0);
\node [draw=none, inner sep = 0] at (0,1.4) {${f(x)}$};
\node [draw=none, inner sep = 0] at (0,-1.4) {$f^2(x) = x$};
\end{tikzpicture}
\hspace{5cm}
&
\begin{tikzpicture}[domain=-3:3, scale=1, knot gap=9, transform canvas={scale=0.65}]
\draw [knot=black, knot gap=0,line width=0.50mm, decoration={markings, mark=at position 0.50  with {\arrow[black,ultra thick]{<}}},
postaction={decorate}](-3,0) to [out=right, in=left, looseness=0.7] (0,1.5) to [out=right, in=left, looseness=0.7] (3,.75);
\draw [knot=black, knot gap=0,line width=0.50mm, decoration={markings, mark=at position 0.50  with {\arrow[black,ultra thick]{<}}},
postaction={decorate}](3,.75) to [out=left, in=right, looseness=0.7] (1,0);
\draw [knot=black, knot gap=0,line width=0.50mm,decoration={markings, mark=at position 0.50  with {\arrow[black,ultra thick]{<}}},
postaction={decorate}](1,0) to [out=right, in=left, looseness=0.7] (3,-.75);
\draw [knot=black, knot gap=0,line width=0.50mm, decoration={markings, mark=at position 0.50  with {\arrow[black,ultra thick]{>}}},
postaction={decorate}](-3,0) to [out=right, in=left, looseness=0.7] (0,-1.5) to [out=right, in=left, looseness=0.7] (3,-.75);
\node [draw=none, inner sep = 0] at (0,-1.8) {$f^4(x) = x$};
\node [draw=none, inner sep = 0] at (3.2,-.35) {${f(x)}$};
\node [draw=none, inner sep = 0] at (3.2,.35) {${f^2(x)}$};
\node [draw=none, inner sep = 0] at (0,1.8) {${f^3(x)}$};
\end{tikzpicture}
\end{tabular}
\]
\vspace{1cm}
\end{example}

The following example shows that Legendrian rack colorings distinguish the trefoil from its positive stabilization.

\begin{example}\label{TrefoilPositiveStab}
Consider the following diagrams of the trefoil and its negative stabilization. A coloring of the diagram of the trefoil by $(X, \tr, f)$ gives the following conditions at the crossings: 
\begin{eqnarray*}
 x \tr f(y) & = & f^2(z),\\
 y \tr f(z) & = & f^2(x),\\
 z \tr f(x) & = & f^2(y)
\end{eqnarray*} 
while a coloring of the diagram of its negative stabilization by $(X, \tr, f)$ gives the following conditions at the crossings: 
\begin{eqnarray*}
 x \tr f(y) & = & f^4(z),\\
 y \tr f(z) & = & f^2(x),\\
 z \tr f(x) & = & f^2(y).
 \end{eqnarray*}
Now by choosing $(X, \tr, f)$ to be the Legendrian rack given in Example~\ref{nonq}, the system of equations for the trefoil has a solution $x=y=z$, while the system of equations for its positive stabilization has no solution, thus the two knots are distinguished by their sets of colorings. 
\vspace{.7cm}
\[
\begin{tabular}{ll}
\begin{tikzpicture}[domain=-3:3, range=-3:3, scale=1, knot gap=9, transform canvas={scale=.65}]
\begin{knot}[consider self intersections, flip crossing=2]
\draw[line width=0.50mm, decoration={markings, mark=at position .1 with {\arrow{>}}}, decoration={markings, mark=at position .9 with {\arrow{>}}}, postaction={decorate}] 
(-3,0) to [out=right, in=left, looseness=0.7] (-1,1.5) to [out=right, in=left, looseness=0.7] (1,.65)to [out=left, in=right, looseness=0.7] (-1,-.65) to [out=right, in=left, looseness=0.7, knot gap =9] (1,-1.5) to [out=right, in=left, looseness=0.7] (3,0) to [out=left, in=right, looseness=0.7] (1,1.5) to [out=left, in=right, looseness=0.7] (-1,.65)to [out=right, in=left, looseness=0.7] (1,-.65) to [out=left, in=right, looseness=0.7] (-1,-1.5) to [out=left, in=right, looseness=0.7] (-3,0);
\end{knot}
\node [draw=none, inner sep = 0] at (-2.5,1) {$f(y)$};
\node [draw=none, inner sep = 0] at (2.5,1) {${f^2(z)}$};
\node [draw=none, inner sep = 0] at (-.8,1.1) {${x}$};
\node [draw=none, inner sep = 0] at (.8,1.1) {${f(y)}$};
\node [draw=none, inner sep = 0] at (.8,.25) {${f^2(y)}$};
\node [draw=none, inner sep = 0] at (-.8,.25) {${f(x)}$};
\node [draw=none, inner sep = 0] at (.8,-1.05) {${f^2(x)}$};
\node [draw=none, inner sep = 0] at (-.8,-.4) {${z}$};
\node [draw=none, inner sep = 0] at (2.5,-1) {${f(z)}$};
\node [draw=none, inner sep = 0] at (-2.5,-1) {${y}$};
\end{tikzpicture} 
\hspace{5cm}
&
\begin{tikzpicture}[domain=-4:4, scale=1, knot gap=9, transform canvas={scale=.65}]
\begin{knot}[consider self intersections, flip crossing=2]
\draw[line width=0.50mm, decoration={markings, mark=at position .1 with {\arrow{>}}}, decoration={markings, mark=at position .9 with {\arrow{>}}}, postaction={decorate}] 
(-3,0) to [out=right, in=left, looseness=0.7] (-1,1.5) to [out=right, in=left, looseness=0.7] (1,.65)to [out=left, in=right, looseness=0.7] (-1,-.65) to [out=right, in=left, looseness=0.7, knot gap =9] (1,-1.5) to [out=right, in=left, looseness=0.7] (2.9,-.65) to [out=left, in=right, looseness=0.7] (2,0)to [out=right, in=left, looseness=0.7] (2.9,.65)to [out=left, in=right, looseness=0.7] (1,1.5) to [out=left, in=right, looseness=0.7] (-1,.65)to [out=right, in=left, looseness=0.7] (1,-.65) to [out=left, in=right, looseness=0.7] (-1,-1.5) to [out=left, in=right, looseness=0.7] (-3,0);
\end{knot}
\node [draw=none, inner sep = 0] at (-2.5,1) {$f(y)$};
\node [draw=none, inner sep = 0] at (2.5,1.2) {${f^4(z)}$};
\node [draw=none, inner sep = 0] at (-.8,1.1) {${x}$};
\node [draw=none, inner sep = 0] at (.8,1.1) {${f(y)}$};
\node [draw=none, inner sep = 0] at (.8,.25) {${f^2(y)}$};
\node [draw=none, inner sep = 0] at (-.8,.25) {${f(x)}$};
\node [draw=none, inner sep = 0] at (.8,-1.05) {${f^2(x)}$};
\node [draw=none, inner sep = 0] at (-.8,-.4) {${z}$};
\node [draw=none, inner sep = 0] at (2.5,-1.2) {${f(z)}$};
\node [draw=none, inner sep = 0] at (-2.5,-1) {${y}$};
\node [draw=none, inner sep = 0] at (3,-.25) {${f^2(z)}$};
\node [draw=none, inner sep = 0] at (2,.6) {${f^3(z)}$};
\end{tikzpicture}
\end{tabular}
\]
\vspace{1cm}
\end{example}

The following example distinguishes between the trefoil and its positive stabilization.
\begin{example}\label{TrefoilNegativeStab}
Consider the following diagrams of the trefoil and its positive stabilization. 
\vspace{.7cm}
\[
\begin{tabular}{ll}
\begin{tikzpicture}[domain=-3:3, range=-3:3, scale=1, knot gap=9, transform canvas={scale=.65}]
\begin{knot}[consider self intersections, flip crossing=2]
\draw[line width=0.50mm, decoration={markings, mark=at position .1 with {\arrow{>}}}, decoration={markings, mark=at position .9 with {\arrow{>}}}, postaction={decorate}] 
(-2.9,0) to [out=right, in=left, looseness=0.7] (-1,1.5) to [out=right, in=left, looseness=0.7] (1,.65)to [out=left, in=right, looseness=0.7] (-1,-.65) to [out=right, in=left, looseness=0.7, knot gap =9] (1,-1.5) to [out=right, in=left, looseness=0.7] (3,0) to [out=left, in=right, looseness=0.7] (1,1.5) to [out=left, in=right, looseness=0.7] (-1,.65) to [out=right, in=left, looseness=0.7] (1,-.65) to [out=left, in=right, looseness=0.7] (-1,-1.5) to [out=left, in=right, looseness=0.7] (-2.9,0) ;
\end{knot}
\node [draw=none, inner sep = 0] at (-2.5,1) {$f(y)$};
\node [draw=none, inner sep = 0] at (2.5,1) {${f^2(z)}$};
\node [draw=none, inner sep = 0] at (-.8,1.1) {${x}$};
\node [draw=none, inner sep = 0] at (.8,1.1) {${f(y)}$};
\node [draw=none, inner sep = 0] at (.8,.25) {${f^2(y)}$};
\node [draw=none, inner sep = 0] at (-.8,.25) {${f(x)}$};
\node [draw=none, inner sep = 0] at (.8,-1.05) {${f^2(x)}$};
\node [draw=none, inner sep = 0] at (-.8,-.4) {${z}$};
\node [draw=none, inner sep = 0] at (2.5,-1) {${f(z)}$};
\node [draw=none, inner sep = 0] at (-2.5,-1) {${y}$};
\end{tikzpicture} 
\hspace{5cm}
&
\begin{tikzpicture}[domain=-3:3, range=-3:3, scale=1, knot gap=9, transform canvas={scale=.65}]
\begin{knot}[consider self intersections, flip crossing=2]
\draw[line width=0.50mm, decoration={markings, mark=at position .1 with {\arrow{>}}}, decoration={markings, mark=at position .9 with {\arrow{>}}}, postaction={decorate}] 
(-3,0) to [out=right, in=left, looseness=0.7] (-1,1.5) to [out=right, in=left, looseness=0.7] (1,.65) to [out=left, in=right, looseness=0.7] (-1,-.65) to [out=right, in=left, looseness=0.7] (1,-1.5) to [out=right, in=left, looseness=0.7] (3,0)to 
[out=left, in=right, looseness=0.7] (1,1.5) to [out=left, in=right, looseness=0.7] (-1,1) to [out=right, in=left, looseness=0.7] (0,.6)  to [out=left, in=right, looseness=0.7] (-1,.2) to [out=right, in=left, looseness=0.7] (1,-.65) to [out=left, 
in=right, looseness=0.7] (-1,-1.5) to [out=left, in=right, looseness=0.7] (-3,0);
\end{knot}

\node [draw=none, inner sep = 0] at (-2.5,1) {$f(y)$};
\node [draw=none, inner sep = 0] at (2.5,1) {${f^2(z)}$};
\node [draw=none, inner sep = 0] at (-.8,1.3) {\small ${x}$};
\node [draw=none, inner sep = 0] at (.8,1.1) {${f(y)}$};
\node [draw=none, inner sep = 0] at (.8,.25) {${f^2(y)}$};
\node [draw=none, inner sep = 0] at (-.8,.7) {\tiny ${f(x)}$};
\node [draw=none, inner sep = 0] at (-.3,.25) {\tiny ${f^2(x)}$};
\node [draw=none, inner sep = 0] at (.8,-.3) {\tiny ${f^3(x)}$};
\node [draw=none, inner sep = 0] at (.8,-1) {\tiny ${f^4(x)}$};
\node [draw=none, inner sep = 0] at (-.8,-.4) {${z}$};
\node [draw=none, inner sep = 0] at (2.5,-1) {${f(z)}$};
\node [draw=none, inner sep = 0] at (-2.5,-1) {${y}$};
\end{tikzpicture} 
\end{tabular}
\]
\vspace{1cm}

As in the previous example, a coloring of the diagram of the trefoil by $(X, \tr, f)$ gives the following conditions at the crossings: 
\begin{eqnarray*}
 x \tr f(y) & = & f^2(z),\\
 y \tr f(z) & = & f^2(x),\\
 z \tr f(x) & = & f^2(y)
\end{eqnarray*} 
while a coloring of the diagram of its negative stabilization by $(X, \tr, f)$ gives the following conditions at the crossings: 
\begin{eqnarray*}
 x \tr f(y)   & = & f^2(z),\\
 y \tr f(z)   & = & f^4(x),\\
 z \tr f^3(x) & = & f^2(y).
 \end{eqnarray*}
Now by choosing $(X, \tr, f)$ to be the one given in the top right corner of the chart in Example~\ref{3elements}, that is $x \tr y= x+1$ and $f=(123)$, the system of equations for the trefoil has a solution with $x=1, y=2$ and $z=3$, while this is not a solution to the system of equations for its positive stabilization, thus the two knots are distinguished by their sets of colorings.
\end{example}

The following example distinguishes between connected sums of Legendrian knots.

\begin{example}\label{ConnectedSum}
Lets call the knot diagrams on the left and on the right of the Figure in Example~\ref{TrefoilPositiveStab} respectively $K_1$ and $K_2$.  Now we use the following $4$ element rack with the map $f$ to distinguish the two connected sums $K_1 \# K_1$ and $K_1 \# K_2$. 
\[
\left[\begin{array}{rrrr}
2 & 2 & 2 & 2 \\
1 & 1 & 1 & 1 \\
4 & 4 & 4 & 4 \\
3 & 3 & 3 & 3
\end{array}\right] \;\; \text{and the map}\; f=(1423).
\]
The coloring of the connected sum $K_1 \# K_1$ 
\vspace{1cm}
\[
\begin{tikzpicture}[domain=-3:3, scale=1, knot gap=9, transform canvas={scale=.65}]
\begin{knot}[consider self intersections, flip crossing=2, decoration={markings, mark=at position .1 with {\arrow{>}}}, decoration={markings, mark=at position .9 with {\arrow{>}}}]
\draw[line width=0.50mm,->] (1,1.5) to [out=right, in=left, looseness=0.7] (1,1.5) to [out=right, in=left, looseness=0.7] (3,.65)to [out=left, in=right, looseness=0.7] (1,-.65) to [out=right, in=left, looseness=0.7, knot gap =9] (3,-1.5) to [out=right, in=left, looseness=0.7] (5,0) to [out=left, in=right, looseness=0.7] (3,1.5) to [out=left, in=right, looseness=0.7] (1,.65)to [out=right, in=left, looseness=0.7] (3,-.65) to [out=left, in=right, looseness=0.7] (1,-1.5); 
\draw[line width=0.50mm] (1,1.5) to [out=left, in=right, looseness=0.7] (-1,1.5);
\draw[line width=0.50mm](-1,1.5) to [out=left, in=right, looseness=0.7]  (-3,.65)to [out=right, in=left, looseness=0.7] (-1,-.65)to [out=left, in=right, looseness=0.7] 
(-3,-1.5)to [out=left, in=right, looseness=0.7] (-5,0)to [out=right, in=left, looseness=0.7] (-3,1.5)to [out=right, in=left, looseness=0.7] (-1,.65)to 
[out=left, in=right, looseness=0.7] (-3,-.65) to [out=right, in=left, looseness=0.7] (-1,-1.5) to [out=right, in=left, looseness=0.7] (1,-1.5);
\end{knot}
\node [draw=none, inner sep = 0] at (-4,1.6) {${f(z)}$};
\node [draw=none, inner sep = 0] at (0,1.7) {${u}$};
\node [draw=none, inner sep = 0] at (4,1.6) {${f^2(w)}$};
\node [draw=none, inner sep = 0] at (-3,1) {${f^2(x)}$};
\node [draw=none, inner sep = 0] at (1,1) {${v}$};
\node [draw=none, inner sep = 0] at (-1.4,.3) {${z}$};
\node [draw=none, inner sep = 0] at (-1.2,-.15) {${f(x)}$};
\node [draw=none, inner sep = 0] at (-1.4,-.95) {${x}$};
\node [draw=none, inner sep = 0] at (-2.8,-.15) {${f(y)}$};
\node [draw=none, inner sep = 0] at (0,-1.7) {${y}$};
\node [draw=none, inner sep = 0] at (-4,-1.6) {${f^2(z)}$};
\node [draw=none, inner sep = 0] at (3,.3) {${f(u)}$};
\node [draw=none, inner sep = 0] at (1,.3) {${f(v)}$};
\node [draw=none, inner sep = 0] at (3,-1) {${f^2(v)}$};
\node [draw=none, inner sep = 0] at (1,-.4) {${w}$};
\node [draw=none, inner sep = 0] at (4,-1.6) {${f(w)}$};
\end{tikzpicture}
\]
\vspace{1cm}

gives the following equations: 
\[
	\begin{array}{lrclrc}
	x \tr y & = & f^2(z), & z \tr f(x)& = & f(y),  \\
	u \tr f(z) & = & f^2(x), & v \tr u & = & f^2(w), \\
	w \tr f(v)& = & f(u), & y \tr f(w) &= & f^2(v).
	\end{array} 
\]
Axiom (III) of Definition~\ref{LegendleDef} simplifies this system to become: 
 	\[ 
 	\begin{array}{rclrcl}
 	x \tr y & = & f^2(z), & z \tr x & = &f(y),  \\
 	u \tr z& = & f^2(x), & v \tr u & = & f^2(w), \\
 	w \tr v& = & f(u), & y \tr w & = & f^2(v).
 	\end{array} 
 	\]
We prove that this system doesn't have a solution:  Let $\sigma=(12)(34)$ be the permutation on $\{ 1,2,3,4\}$ so that the rack operation becomes $x \tr y =\sigma(x), \forall x,y$.  First notice that $f^2=\sigma$ and thus the maps $f$ and $\sigma$ commute.  Then the first equation, $x \tr y=f^2(z)$, of the system gives $z=x$, while the equation $z \tr x=f(y)$ implies $y=f(z)=f(x)$.  The equation $u \tr z=f^2(x)$ gives $u=x$, while the equation $v \tr u=f^2(w) $ implies $v=w$.  The equation $w \tr v=f(u)$ gives $f(w)=u$ and the equation $y \tr w=f^2(v)$ implies $y=v$, thus $x=f(y)$, implying $x=f^2(x)$ but this is impossible since $f$ has no \emph{fixed} point.  
Now the coloring of $K_1 \# K_2$ in the figure 
\vspace{1cm}
\[
\begin{tikzpicture}[domain=-8:8, scale=1, knot gap=9, transform canvas={scale=.65}]
\begin{knot}[consider self intersections, flip crossing=2]
\draw[line width=0.50mm, ->] (1,1.5) to [out=right, in=left, looseness=0.7] (3,.65)to [out=left, in=right, looseness=0.7] (1,-.65) to [out=right, in=left, looseness=0.7, knot gap =9] (3,-1.5) to [out=right, in=left, looseness=0.7] (4.9,-.65) to [out=left, in=right, looseness=0.7] (4,0)to [out=right, in=left, looseness=0.7] (4.9,.65)to [out=left, in=right, looseness=0.7] (3,1.5) to [out=left, in=right, looseness=0.7] (1,.65)to [out=right, in=left, looseness=0.7] (3,-.65) to [out=left, in=right, looseness=0.7] (1,-1.5); 
\draw[line width=0.50mm] (1,1.5) to [out=left, in=right, looseness=0.7] (-1,1.5) to [out=left, in=right, looseness=0.7]  (-3,.65)to [out=right, in=left, looseness=0.7] (-1,-.65)to [out=left, in=right, looseness=0.7] 
(-3,-1.5)to [out=left, in=right, looseness=0.7] (-5,0)to [out=right, in=left, looseness=0.7] (-3,1.5)to [out=right, in=left, looseness=0.7] (-1,.65)to 
[out=left, in=right, looseness=0.7] (-3,-.65) to [out=right, in=left, looseness=0.7] (-1,-1.5) to [out=right, in=left, looseness=0.7] (1,-1.5);
\end{knot}
\node [draw=none, inner sep = 0] at (-4,1.6) {${f(z)}$};
\node [draw=none, inner sep = 0] at (0,1.7) {${u}$};
\node [draw=none, inner sep = 0] at (4,1.6) {${f^4(w)}$};
\node [draw=none, inner sep = 0] at (-3,1) {${f^2(x)}$};
\node [draw=none, inner sep = 0] at (1,1) {${v}$};
\node [draw=none, inner sep = 0] at (-1.4,.3) {${z}$};
\node [draw=none, inner sep = 0] at (-1.2,-.15) {${f(x)}$};
\node [draw=none, inner sep = 0] at (-1.4,-.95) {${x}$};
\node [draw=none, inner sep = 0] at (-2.8,-.15) {${f(y)}$};
\node [draw=none, inner sep = 0] at (0,-1.7) {${y}$};
\node [draw=none, inner sep = 0] at (-4,-1.6) {${f^2(z)}$};
\node [draw=none, inner sep = 0] at (3,.3) {${f(u)}$};
\node [draw=none, inner sep = 0] at (1,.3) {${f(v)}$};
\node [draw=none, inner sep = 0] at (3,-1) {${f^2(v)}$};
\node [draw=none, inner sep = 0] at (1,-.4) {${w}$};
\node [draw=none, inner sep = 0] at (4,-1.6) {${f(w)}$};
\node [draw=none, inner sep = 0] at (5,-.2) {${f^2(w)}$};
\node [draw=none, inner sep = 0] at (4,.5) {${f^3(w)}$};
\end{tikzpicture}
\]
\vspace{1cm}

gives the following equations: 
		\[ 
		\begin{array}{rclrcl}
	x \tr y& = & f^2(z), & z \tr f(x) & = & f(y),\\
	u \tr f(z)& = & f^2(x), & v \tr u & = & f^4(w),\\
	w \tr f(v)& = & f(u), &	y \tr f(w)& = & f^2(v).
		\end{array}
		\]		
 Axiom (III) of Definition~\ref{LegendleDef} simplifies this set of to become: 
 \[
\begin{array}{rclrcl}
 x \tr y & = & f^2(z), &
 z \tr x & = & f(y),\\
 u \tr z & = & f^2(x), &
 v \tr u & = & f^4(w),\\
 w \tr v & = & f(u),&
 y \tr w & = & f^2(v).
 \end{array}
 \]
One checks easily that setting $x=z=u=1$, $y=v=4$ and $w=3$ give a solution of this system of equations and thus a coloring of  $K_1 \# K_2$.  Now since $K_1 \# K_1$ doesn't have a coloring, we conclude that the two Legendrian knots $K_1 \# K_1$ and $K_1 \# K_2$ are distinct.
\end{example}

\noindent
\textsc{Department of Mathematics, \\
	Hamilton College, \\
	198 College Hill Rd.,\\
	Clinton, NY 13323,}\\


\noindent
\textsc{Department of Mathematics, \\
University of South Florida, \\
4202 E Fowler Ave., \\
Tampa, FL 33620}\\

\noindent
\textsc{Department of Mathematics, \\
Claremont McKenna College, \\
850 Columbia Ave., \\
Claremont, CA 91711}

\begin{thebibliography}{10}


\bibitem{KA}{D. Bar-Natan, Ed. The Knot Atlas, 
\texttt{http://katlas.math.toronto.edu/wiki/Main\underline{\ }Page}}


 \bibitem{CEGS}{S. Carter,  M. Elhamdadi, M. Gra\~na and M. Saito.  
Cocycle knot invariants from quandle modules and generalized
              quandle homology. 
\textit{Osaka J. Math.}  \textbf{42}  (2005),  no.3, 499--541.}


\bibitem{CES1}{S. Carter,  M. Elhamdadi, and M. Saito.  
Homology theory for the set-theoretic Yang-Baxter equation and knot invariants 
from generalizations of quandles. 
\textit{   Fund. Math.}  \textbf{184}  (2004),  31--51.}

\bibitem{CENS}{S. Carter,  M. Elhamdadi, M. Nikiforou and M. Saito. 
Extensions of quandles and cocycle knot invariants.
\textit{J. Knot Theory Ramifications}  \textbf{12}  (2003), no 6, 725-738.} 


\bibitem{Ceniceros}{J. Ceniceros.  
	On Braids, Branched Covers and Transverse Invariants.
	\textit{PhD Thesis, Louisiana State University}  (2017).}

\bibitem{CN1}{J. Ceniceros and S. Nelson.  
Virtual Yang-Baxter Cocycle Invariants.
\textit{Trans. Amer. Math. Soc.}  \textbf{361}  (2009),  no. 10, 5263--5283.}

\bibitem{CH}{Yu. Chekanov, Differential algebra of Legendrian links, \textit{Invent.Math.} \textbf{150} (2002), no. 3, 441–483.}

\bibitem{CES}{Z. Cheng, M. Elhamdadi and B. Shekhtman.  
On the classification of topological quandles.
\textit{Topology Appl.}  \textbf{248}  (2018),  64--74}.

\bibitem{CN}{W. Chongchitmate and L. Ng.
	An atlas of Legendrian knots. 
	\textit{ Exp. Math. 22. }  \textbf{361}  (2013),  no. 1, 26-37.}
	
	\bibitem{EFT}{M. Elhamdadi, N. Fernando, and B. Tsvelikhovskiy,  
Ring Theoretic Aspects of Quandles.
\textit{ J. Algebra ,}  \textbf{526}  (2019), 166-187. }

\bibitem{EM}{M. Elhamdadi and E. Moutuou.  
Foundations of topological racks and quandles.
\textit{ J. Knot Theory Ramifications} 25 (2016), no. 3, 1640002, 17 pp.}


\bibitem{EMR}{M. Elhamdadi, J. MacQuarrie, and R. Restrepo.  
Automorphism groups of quandles.
\textit{ J. Algebra Appl.,}  \textbf{11}  (2012), no. 1,  1250008, 9. }


\bibitem{ENbook}{M. Elhamdadi and S. Nelson.  
Quandles-an introduction to the algebra of knots.
\textit{ Student Mathematical Library, American Mathematical Society, Providence, RI,}  \textbf{74}  (2015), 245 pp. }

\bibitem{EN0}{M. Elhamdadi and S. Nelson.  
	$N$-degeneracy in rack homology and link invariants.
	\textit{ Hiroshima Math. J.} \textbf{42}, (2012), no. 1, 127--142}. 

\bibitem{EN}{Y. Eliashberg.  
	Invariants in contact topology. Proceedings of the International Congress of Mathematicians, Vol. II (Berlin, 1998). 
	\textit{ Doc. Math. 1998, Extra,  } \textbf{Vol. II}, (1998), 327-338}. 

\bibitem{EN1}{Y. Eliashberg and M. Fraser.  
	Topologically trivial Legendrian knots.
	\textit{ J. Symplectic Geom. 7,  } \textbf{ no. 2}, (2009),  77-127}. 

\bibitem{EN2}{Y. Eliashberg and M. Fraser.  
	Classification of topologically trivial Legendrian knots.
	\textit{ CRM Proc. Lecture Notes, 15, Amer. Math. Soc., Providence, RI. } \textbf{}, (1998),  17-51}.

\bibitem{EN3}{Y. Eliashberg, A. Givental and H. Hofer, Introduction to symplectic field theory,
\textit{Geom. Funct. Anal.}(Special Volume, Part II) (2000), 560–673 GAFA 2000 (Tel Aviv,
1999).}

\bibitem{Etnyre} J. Etnyre,
{\it Legendrian and transversal knots.}
Handbook of knot theory, 105-185, Elsevier B. V., Amsterdam, 2005. 

\bibitem{Etnyre1} J. Etnyre,
{\it Introductory Lectures on Contact Geometry}
\textit{Proc. Sympos. Pure Math.} \textbf{71} (2003), 81-107.


\bibitem{FR} R. Fenn, C. Rourke,
 \textit{Racks and links in codimension two},
 J. Knot Theory Ramifications, \textbf{1} (1992)  343--406.

 \bibitem{Geiges} H. Geiges,
  {\it An introduction to contact topology}, 2008, Cambridge University Press.


\bibitem{Joyce}{ D. Joyce,
 A classifying invariant of knots, the knot quandle.
\textit{J. Pure Appl. Alg.}, \textbf{23}, 37--65. 1982.}


\bibitem{KulPra}{ D. Kulkarni and V. Prathamesh, 
\textit{ On rack Invariants of Legendrian Knots},  (2017),	arXiv:1706.07626 }


\bibitem{AdamN} {A. Navas, and S. Nelson. On symplectic quandles.
\textit{Osaka J. Math.}  \textbf{45}  (2008),  no. 4, 973--985.}

\bibitem{Matveev} S. Matveev, 
{\it Distributive groupoids in knot theory,} (Russian) Mat. Sb. (N.S.)
119 (161) (1982), no. 1, 78--88, 160.

\bibitem{N2}{S. Nelson. 
A polynomial invariant of finite quandles.  
\textit{J. Algebra Appl.}   \textbf{7}  (2008),  no. 2, 263--273.}

\bibitem{Sam} S. Nelson,
{\it Classification of finite Alexander quandles,} Topology Proceedings 27 (2003), 245--258.


\bibitem{Oz}{P. S. Ozsv\'ath.; A. I. Stipsicz; Z. Szab\'o, 
{\it  Grid homology for knots and links}
Mathematical Surveys and Monographs, 208. American Mathematical Society, Providence, RI, 2015}

\bibitem{Oz1} {P. S. Ozsv\'ath, Z. Szab\'o, and D. Thurston, Legendrian knots, transverse knots and combinatorial Floer homology, \textit{Geom. Topol.} \textbf{12} (2008), no. 2, 941–980, Zbl 1144.57012, MR 2403802
(2009f:57051).}


\bibitem{Rubin}{R. L. Rubinsztein, 
{\it  Topological quandles and invariants of links}
J. Knot Theory Ramifications 16 (6) (2007) 789--808. }

\bibitem{Sab}{J. M. Sabloff, 
{\it  What is  a Legendrian knot?}
Notices Amer. Math. Soc. 56 (2009), no. 10, 128-1284. }


\bibitem{Taka} {M. Takasaki, 
{\it Abstraction of symmetric transformation,} (in Japanese),
Tohoku Math. J. 49 (1942/3), 145--207.}

\end{thebibliography}
\end{document}